\documentclass[12pt]{amsart}
\usepackage{amsmath,amsthm,amsfonts,amssymb,times,latexsym,mathabx,url, eufrak}

\newtheorem{theorem}{Theorem}
\newtheorem{prop}{Proposition}[section]
\newtheorem{lem}{Lemma}[section]

\newtheorem*{df}{Definition}
\newtheorem{rem}{Remark}[section]
\numberwithin{equation}{section}

\DeclareMathOperator{\AP}{AP}

\makeatletter
\renewcommand{\pmod}[1]{\allowbreak\mkern7mu({\operator@font mod}\,\,#1)}
\makeatother

\renewcommand{\d}{\delta}
\newcommand{\eps}{\varepsilon}
\newcommand{\D}{\mathcal D}
\newcommand{\E}{\mathbb{E}}

\renewcommand{\l}{\lambda}

\newcommand{\N}{\mathbb{N}}
\renewcommand{\o}{\omega}

\renewcommand{\P}{\mathbb P}

\newcommand{\s}{\sigma}

\newcommand{\U}{\mathcal U}
\newcommand{\V}{\mathcal V}
\newcommand{\Z}{\mathbb{Z}}

\renewcommand{\b}{\mathbf b}
\newcommand{\e}{\mathbf e}
\newcommand{\n}{\mathbf n}
\renewcommand{\S}{\mathbf S}
\newcommand{\bAP}{\mathbf{AP}}
\newcommand{\bl}{\boldsymbol{\lambda}}
\newcommand{\Q}{\mathcal Q}
\newcommand{\bQ}{\boldsymbol{\mathcal Q}}
\newcommand{\mE}{\mathcal E}
\newcommand{\bmE}{\boldsymbol{\mathcal E}}
\newcommand{\mF}{\mathcal F}
\newcommand{\bmF}{\boldsymbol{\mathcal F}}
\newcommand{\HH}{\mathfrak H}

\renewcommand{\leq}{\leqslant}
\renewcommand{\geq}{\geqslant}

\begin{document}

\title{Prime avoiding numbers is a basis of order $2$}
\author{Mikhail R. Gabdullin}
\date{}
\address{Steklov Mathematical Institute,
		Gubkina str., 8, Moscow, Russia, 119991}
\email{gabdullin.mikhail@yandex.ru}

\begin{abstract}
For a positive integer $n$, we denote by $F(n)$ the distance from $n$ to the nearest prime number. We prove that every sufficiently large positive integer $N$ can be represented as the sum $N=n_1+n_2$, where
$$
F(n_i) \geq (\log N)(\log\log N)^{1/325565},	
$$
for $i=1,2$. This improves the corresponding ``trivial'' statement where only $F(n_i)\gg \log N$ is required.
\end{abstract}

\date{\today}
\maketitle

\section{Introduction}\label{sec1} 

Let $p_n$ be the $n^{th}$ prime and 
$$
G(X)=\max_{p_{n+1}\leq X}(p_{n+1}-p_n)
$$
denote the largest gap between consecutive primes up to $X$. The Prime Number Theorem together with a simple averaging argument implies that $G(X)\geq (1+o(1))\log X$, and Rankin \cite{Rankin} in 1938 was the first to prove the bound of the type
$$
G(X)\geq(c+o(1))\frac{\log X \log\log X\log\log\log\log X}{(\log\log\log X)^2},
$$ 
improving the previous results of Westzynthius \cite{West} and Erd\H{o}s \cite{Erdos1}. Rankin proved the mentioned bound with $c=1/3$, and for about next 80 years this constant was increased many times, the last being $c=2e^{\gamma}$ due to Pintz \cite{Pintz}. In 2016, Ford, Green, Konyagin, Tao \cite{Large1} and independently Maynard \cite{Large2} showed by different approaches that
$c$ can be taken arbitrarily large, giving the affirmative answer for a long-standing conjecture of Erd\H{o}s \cite{Erdos2}. In 2018 all these five authors together, by combining the ideas from \cite{Large1} and \cite{Large2}, made a further breakthrough  \cite{Long} establishing that
\begin{equation}\label{0.1}
G(X)\gg \frac{\log X \log\log X \log\log\log\log X}{\log\log\log X}.
\end{equation}
The expected size of $G(X)$ is of order $(\log X)^2$: see \cite{Cramer} for the precise conjecture of Cram\'er based on a probabilistic model of primes, and \cite{Granv} for its refinement. We note that the best known upper bound for $G(X)$ is $G(X)\ll X^{0.525}$ due to Baker, Harman and Pintz \cite{BHP}, and refer the reader to the paper \cite{Long} for further discussion of the quantity $G(X)$.

In this paper we consider the following additive problem related to the large prime gaps. For a positive integer $n$, let $F(n)$ denote  the distance from $n$ to the nearest prime number; clearly, the maximum value of $F(n)$ taken over all $n\leq X$ has the same order as $G(X)$. Can we prove that any large positive integer $N$ can be represented as the sum of two numbers $n_1$ and $n_2$, where both $F(n_1)$ and $F(n_2)$ are large in terms of $N$? The first instinct to prove such a result may be to use the technique from \cite{Long} (which, in general, follows the strategy from the previous papers, starting from the one of Westzynthius). However, there is a general obstacle which makes this idea unfit for our setting. In that approach, one exploits a ``smooth'' number $m$ which is divided by all small primes up to some (relatively small) $z$ to make sure that the majority of $G(X)$ numbers starting from $m+2$ has a small prime factor; the goal then is to use larger primes to sieve out the remaining numbers and make this procedure as efficient as possible. But if we take $n_1$ and $n_2$ according to this construction, then they are both close to such smooth numbers, and then their sum (which we want to be $N$) is also close to a smooth number, and thus is not arbitrary. Thus one needs to use a completely different method to attack the posed problem.

Firstly, it turns out that some standard technique allows to establish the following ``trivial'' statement. Note that the Prime Number Theorem implies that the average value of $F(n)$ (taken over $n\leq N$) is at least of order $\log N$. 

\begin{prop}\label{prop1}
Every sufficiently large positive integer $N$ can be represented as the sum $N=n_1+n_2$, where $F(n_i)\gg \log N$, $i=1,2$. 	
\end{prop}	

Our goal is to improve the lower bound from this proposition by obtaining a result where $\log N$ is multiplied by some growing function. For a number $\rho\in(0,1)$, we define
\begin{equation}\label{0.2}
C(\rho)=\sup\left\{\d>0: \frac{6\cdot10^{2\d}}{\log(1/(2\d))}<\rho\right\}.
\end{equation}

Our main result is the following.

\begin{theorem}\label{th1} 
Every sufficiently large positive integer $N$ can be represented as the sum \\ $N=n_1+n_2$, where
$$
F(n_i) \geq (\log N)(\log\log N)^{C(1/2)-o(1)},	
$$
for $i=1,2$. 
\end{theorem}

In fact, the proof of Theorem \ref{th1} implies that there are at least $\exp((\log N)^{1-o(1)})$ such representations (because of many ``good'' choices of $\b$ in Theorem \ref{th2}; see Section \ref{sec4}). Note that numerical calculations show that $C(1/2)>1/325565$. 

Theorem \ref{th1} admits the following interpretation. Recall that a set $A\subseteq \N$ is called a basis of order $k$, if every sufficiently large positive integer can be represented as the sum of $k$ summands from $A$. Now consider the set
$$
\Big\{n\in\N : F(n)\geq (\log n)(\log\log n)^\d \Big\}
$$
(informally, the set of ``prime avoiding numbers''). Theorem \ref{th1} then implies that this set is a basis of order $2$ for any $\d<C(1/2)$.


To prove Theorem \ref{th1}, we apply the technique from the recent paper \cite{sieved} of Ford, Konyagin, Maynard, Pomerance, and Tao, where the authors used a hypergraph covering lemma of Pippenger-Spencer type (which was introduced in \cite{Long}) to detect long gaps in general sieved sets. To state their result we need the following definition (the symbol $p$ always denotes a prime number). 

\begin{df}[Sieving System] A sieving system is a collection $\mathcal{I}$ of sets $I_p \subset \Z/p\Z$ of residue classes modulo $p$ for each prime $p$. Moreover, we have the following definitions.
\begin{itemize}
\item (Non-degeneracy) We say that the sieving system is non-degenerate if $|I_p|\leq p-1$ for all $p$.
\item ($B$-Boundedness) Given $B>0$, we say that the sieving system is $B$-bounded if $|I_p| \leq B$ for all primes $p$.
\item (One-dimensionality) We say that the sieving system is one-dimensional if we have the weighted Mertens-type product estimate
$$
\prod_{p\leq x}\left(1-\frac{|I_p|}{p}\right)\sim 
\frac{C_1}{\log x}, \quad (x\to\infty),
$$
for some constant $C_1>0$.

\item ($\rho$-supportedness) Given $\rho>0$, we say that the sieving system system is $\rho$-supported if the
density of primes with $|I_p| \geq 1$ equals $\rho$, that is,
$$
\lim_{x\to\infty}\frac{|{p\leq x: |I_p| \geq 1}|}{x/\log x}
= \rho.
$$
\end{itemize}
\end{df}

The main result of \cite{sieved} is that for such a sieving system, the sieved set 
$$
S_x=S_x(\mathcal{I})=\Z\setminus\bigcup_{p\leq x}I_p
$$
(the set of integers which do not belong to any $I_p$ for all $p\leq x$) contains a gap of size $x(\log x)^{C(\rho)-o(1)}$, where $C(\rho)$ is defined\footnote{Note that in the published journal version of \cite{sieved} there were some inaccuracies throughout the proof which led to the inappropriate definition (1.4) of $C(\rho)$. To fix the proof, one should use our definition (\ref{0.2}) of $C(\rho)$, and this is to appear (at least, on the ArXiV) in the Corrigendum of \cite{sieved}.} in (\ref{0.2}) and the rate of decay in $o(1)$ depends on $\mathcal{I}$. Despite the fact that this general bound applied to the Eratosthenes sieve (that is, the sieving system with $I_p=\{0\}$ for all $p$) yields only a bound
$$
G(X) \gg (\log X)(\log\log X)^{C(1)-o(1)} \gg (\log X)(\log\log X)^{1/835},
$$
which is weaker than (\ref{0.1}), it has the advantage of not dealing with ``smooth'' numbers from the above discussion, and this is crucial for us.

To deduce Theorem \ref{th1}, we also work with the one-dimensional Eratosthenes sieving system; however, we need to treat not one but two sets $S_x-n_1$ and $S_x-N+n_1$ (for some $n_1$) simultaneously; our main goal will be to guarantee the inequality (\ref{2.2}) from Section \ref{sec2}. To do so, we use disjoint sets of (``large'') primes of density $1/2$ (and this is why our result contains the exponent $C(1/2)$) to handle those two sets separately. Fortunately for us, the one-dimensionality is in fact needed only for ``small'' primes, and we are able to make use of it before partitioning the large primes. Except this issue of dealing with two sets simultaneously, our proof actually almost repeats the proof of the main result from \cite{sieved}. However, due to the new technical issues, there are not so many things from \cite{sieved} we can use without any changes, and so we decided to provide the full proof in spite of huge intersection with the text of \cite{sieved}.  

The paper is organized as follows. In Section \ref{sec2} we prove Proposition \ref{prop1} (it is relatively short and based on the ideas we need for our main result); at the same time, we reduce Theorem \ref{th1} to the problem of sieving out two shifts of $S_{x/2}$. In Section \ref{sec3} we give an outline of the next part of the proof and introduce the needed notation. Arguments from Sections \ref{sec4}, \ref{sec5}, \ref{sec6} are analogs of those from Sections 3,4,5 of \cite{sieved}, and our Theorems \ref{th2} and \ref{th3} are modifications of Theorems 2 and 3 of \cite{sieved}. 

\medskip 

\textbf{Acknowledgements.} The author would like to thank Sergei Konyagin for introducing him to this question.

\section{Preliminaries and proof of Proposition \ref{prop1}}\label{sec2}

In this section we provide a proof of Proposition \ref{prop1} to illustrate some parts of the general strategy in a more simple context and also make the first reduction of Theorem \ref{th1}. This proof is actually similar to that of (1.8) of \cite{sieved} (which shows the existence of a gap of length $\gg x$ in $S_x$) with that difference that we again have to handle two sets instead of one. 

\begin{proof}[Proof of Proposition \ref{prop1}]
For a number $x\geq 2$, let 
$$
S_x=\{n\in \Z: n\notequiv0 \pmod{p} \mbox{ for each } p\leq x\}.
$$	
Clearly, $S_x$ is a periodic set with the period $P(x)=\prod_{p\leq x}p$. Let $z=x/2$ and let a number $b'\in \Z/P(z)\Z$ be chosen uniformly at random. We consider the random sets 
$$
A_{b'}:=(S_z-b')\cap[-y,y]
$$
and
$$
A_{N-b'}:=(S_z-N+b')\cap[-y,y],
$$
where $y=\lfloor 0.08x \rfloor$. We have
\begin{multline*}
\E|A_{b'}|=\E\sum_{|n|\leq y}1_{n\in S_z-b'}=\sum_{|n|\leq y}\prod_{p\leq z}\P(b'\notequiv -n\pmod{p})\\
=\sum_{|n|\leq y}\prod_{p\leq z}\left(1-1/p\right)=\frac{(2+o(1))y}{\log z}, 	
\end{multline*}
and similarly
$$
\E|A_{N-b'}|=\frac{(2+o(1))y}{\log z}.
$$
Therefore, if $x$ is large enough, 
$$
\E(|A_{b'}|+|A_{N-b'}|) \leq \frac{5y}{\log x}.
$$
Thus, there is a choice $b'$ modulo $P(z)$ such that 
$$
|A_{b'}|+|A_{N-b'}|\leq \frac{0.4x}{\log x}.
$$ 

Let $P_{z,x}=\prod_{z<p\leq x}p$ and 
$$
S_{z,x}=\{n\in\Z: n\notequiv 0 \pmod p\quad \forall p\in(z,x] \, \}.
$$ 
Now we choose a number $b$ modulo $P(x)$. We set $b\equiv b'\pmod{P(z)}$ and claim that there is a choice $b \pmod{P_{z,x}}$ (let us denote it $b''$) such that
\begin{equation}\label{1.1}
(S_x-b)\cap[-y,y]=(S_x-N+b)\cap [-y,y] = \varnothing. 
\end{equation}
To see that this is possible, note that
$$
S_x-b=\{n\in \Z: n\notequiv -b \pmod p \quad \forall p\leq x\} = (S_z-b')\cap (S_{z,x}-b'');
$$
further, for each element $m\in A_{b'}$ we take a prime $q\in(z,x]$ and define $b\equiv b_q\pmod q$ such that $m\equiv -b_q \pmod q$; so, $m\notin S_{z,x}-b''$ and thus $m\notin S_x-b$.  We do similarly for each $m\in A_{N-b'}$. Since there are $(0.5+o(1))x/\log x$ primes in $(z,x]$ and at most $0.4x/\log x$ survived numbers $m\in A_{b'}\cup A_{N-b'}$, it is possible to make this ``clean-up'' stage. 

To finish the proof, we define $f(n)=\min\{|n-l|: l\in S_x\}$. The equality (\ref{1.1}) then means that  
$$
f(b)\geq y, \quad f(N-b)\geq y
$$
for our choice of $b\pmod{P(x)}$. Now we choose $x\approx \log (N/2)$ maximally so that $P(x)\leq N/2$.
We thus see that it is possible to take that $b$ with $b\in [N/4, 3N/4]$; then $N-b\in [N/4,3N/4]$ as well. Finally, since $\{N^{1/2}<p\leq N\} \subset S_x$, we get
$$
F(b)\geq f(b) \gg x \gg \log N,
$$
and similarly $F(N-b)\gg \log N$. This completes the proof of Proposition \ref{prop1}.
\end{proof}

\medskip 

Now we see that to prove Theorem \ref{th1}, it is enough to show that for any fixed $\d<C(1/2)$ and $y=\lceil x(\log x)^{\d} \rceil$ there exists a choice of $b$ modulo $P(x/2)$ such that
\begin{equation}\label{1.2}
\Big|\Big((S_{x/2}-b)\cup (S_{x/2}-N+b)\Big)\cap [-y,y]\Big| \leq \left(\frac12-\eps\right)\frac{x}{\log x}	
\end{equation}
for some $\eps>0$. Then arguing according to the clean up stage in the above proof, one can easily obtain that both $F(b)$ and $F(N-b)$ are $\gg (\log N)(\log\log N)^{\d}$, and Theorem \ref{th1} will follow. Note that the condition $\d<C(1/2)$ is equivalent to (recall the definition (\ref{0.2}) of $C(\rho)$)
\begin{equation}\label{1.3}
\frac{6\cdot10^{2\d}}{\log(1/(2\d))}<\frac12;
\end{equation}
we will use it in this form in Section \ref{sec5}.

\section{Notation and outline}\label{sec3}

Throughout the proof, we will use positive parameters $K$, $\xi$, $M$ which we describe below; one may think of them as being fixed for most of the time (in fact, it is only the end of Section \ref{sec5} where the exact choice of them is important). The implied constant in $\ll$ and related order estimates may depend on these parameters.  We will rely on probabilistic methods; boldface symbols such as $\S'$, $\bl$, $\n$, etc. will denote the random variables (sets, functions, numbers, etc.), and the corresponding non-boldface symbols $S'$, $\l$, $n$ will denote deterministic counterparts of these variables.

For a fixed $\d>0$ with (\ref{1.3}), we define
\begin{equation}\label{2.1}
y=\lceil x(\log x)^{\d} \rceil
\end{equation}
and
\begin{equation}\label{2.2}
z=\frac{y\log\log x}{(\log x)^{1/2}}.
\end{equation}
Let $\xi>1$ be a real number (which we will finally choose to be close to $1$), and
$$
\HH=\left\{ H\in \{1,\xi, \xi^2,...\}: \frac{2y}{x} \leq H\leq \frac{y}{\xi z}\right\}
$$
so that each $H$ obeys
\begin{equation}\label{2.3}
2(\log x)^{\d}\leq H\leq \frac{y}{z} = \frac{(\log x)^{1/2}}{\log\log x}.
\end{equation}
For each $H$ and $i\in\{1,3\}$, let $\Q_{H,i}$ be the set of primes $i\pmod{4}$ in $(y/(\xi H),y/H]$. Note that
\begin{equation}\label{2.4}
|\Q_{H,i}|\sim (1-1/\xi)\frac{y}{2H\log x}
\end{equation}
whenever $x$ is large enough in terms of $\xi$. Let
$$
\Q=\bigcup_{H\in\HH}\left(\Q_{H,1}\cup\Q_{H,3}\right),
$$
and for each $q\in \Q$, let $H_q$ be the unique $H$ such that $q\in \Q_{H,1}\cup\Q_{H,3}$, which is equivalent to
$$
\frac{y}{\xi H}<q\leq \frac{y}{H}.
$$
Let also $M$ be a number with
\begin{equation}\label{2.5}
6<M\leq 7.	
\end{equation}

As in the previous section, we will use the notation
$$
S_z=\{n\in\Z: n\notequiv 0 \pmod{p} \mbox{ for all } p\leq z\}
$$
and
$$
S_{z,u}=\{n\in\Z: n\notequiv 0 \pmod{p} \mbox{ for all } z<p\leq u\}.
$$
We adopt the abbreviations
\begin{equation}\label{2.6}
P=P(z)=\prod_{p\leq z}p, \quad \s=\s(z)=\prod_{p\leq z}\left(1-\frac1p\right), \quad \S'=S_z-\b, \quad \S''=S_z-N+\b,  
\end{equation}
where $\b$ is a residue class chosen uniformly at random from $\Z/P(z)\Z$; so, both $\S'$ and $\S''$ are random shifts of $S_z$. For a fixed $H\in \HH$, we also define
\begin{equation}\label{2.7}
P_1=\prod_{p\leq H^M}p, \quad \s_1=\s(H^M), \quad \b_1\equiv \b \pmod {P_1}, \quad \S'_1=S_{H^M}-\b_1, \quad \S''_1=S_{H^M,z}-N+\b_1,  
\end{equation}
and
\begin{equation}\label{2.8}
P_2=\prod_{H^M<p\leq z}p, \quad \s_2=\s(H^M,z), \quad \b_2\equiv \b \pmod {P_2}, \quad \S'_2=S_{H^M}-\b_2, \quad \S''_2=S_{H^M,z}-N+\b_2,  
\end{equation}
where $\s(H^M,z)=\prod_{H^M<p\leq z}\left(1-\frac1p\right)$. Obviously, for each $H\in\HH$,
\begin{equation}\label{2.85}
P=P_1P_2, \quad \s=\s_1\s_2, \quad \S'=\S'_1\cap \S'_2, \quad \S''=\S''_1\cap \S''_2.
\end{equation}
Note that all the quantities defined in (\ref{2.7}) and (\ref{2.8}) depend on $H$ and $M$; however, we will not indicate this dependence for the brevity (the values of $H$ and $M$ will always be clear from the context). 

Finally, we define
\begin{equation}\label{2.9}
\bAP'(KH_q; q,n)=\{n+qh: 1\leq h\leq KH_q\}\cap \S_1',
\end{equation}
\begin{equation}\label{2.10}
\bAP''(KH_q; q,n)=\{n+qh: 1\leq h\leq KH_q\}\cap \S_1'',
\end{equation}
where $K$ is a positive integer which will be chosen large enough, and, for $q\equiv1\pmod4$,
\begin{equation}\label{2.11}
\bl(H_q; q,n)=
\begin{cases}
\s_2^{-|\bAP'(KH_q,q,n)|}, &\mbox{if } \bAP'(KH_q,q,n) \subset \S'_2;\\ 
0, &\mbox{otherwise}.		
\end{cases}
\end{equation}
and, for $q\equiv3\pmod4$,
\begin{equation}\label{2.12}
\bl(H_q; q,n)=
\begin{cases}
\s_2^{-|\bAP''(KH_q,q,n)|}, &\mbox{if } \bAP''(KH_q,q,n) \subset \S''_2;\\ 
0, &\mbox{otherwise}.		
\end{cases}
\end{equation}
So, for each $q\in\Q$, the weights $\bl(H_q;q,n)$ are random functions which depend on $\b$. 

\bigskip 

Now we give a brief outline of the proof. As in \cite{sieved}, there are three main steps:

\begin{enumerate}
\item (Uniform random stage) We choose $\b$ modulo $P(z)$ uniformly at random; this is equivalent to choosing $b \pmod p$ randomly with uniform probability, independently for each $p\leq z$. Then, first of all, we can easily guarantee that both sets $\S'\cap[-y,y]$ and $\S''\cap[-y,y]$ have size about $2y/\log z$ (see Remark \ref{rem4.1} below). We also show that with high probability the sets $\S_1',\S_2',\S_1'',\S_2''$ behave as we need them to for all scales $H\in\HH$. 

\item (Greedy stage) Having chosen an appropriate $b\pmod{P(z)}$, we continue sieving out the sets $S'\cap[-y,y]$ and $S''\cap[-y,y]$. They have small intersection, so we need to work with both of them separately using disjoint subsets of ``large'' (those between $z$ and $x/2$) primes $\{q\in\Q: q\equiv1\pmod4\}$ and $\{q\in\Q: q\equiv3\pmod4\}$ of density $1/2$, respectively. To establish (\ref{1.2}), we select $b$ modulo $P(z,x/2)$ randomly, but dependent on the choice of $b$ modulo $P(z)$. Slightly more precisely, for each prime $q\in(z,x/2]$ with $q\equiv1\pmod4$, we will select $b\equiv b_q \pmod q$ so that $\{b_q+kq: k\in\Z\}$ knocks out nearly as many
elements of $(S_z-b) \cap [-y,y]$ as possible; we do the same with primes $q\in(z,x/2]$, $q\equiv3\pmod4$, to sieve out almost all of $(S_z-N+b)\cap[-y,y]$. This can be done using the so-called hypergraph covering theorem (Lemma \ref{lem3.2} below).

\item (Clean up stage) Finally, as we saw in the previous section, one can use the remaining primes in $(x/2,x]$ to ``kill'' all the numbers from both $S'\cap[-y,y]$ and $S''\cap[-y,y]$ survived after the greedy stage, and this actually completes the proof. 
\end{enumerate}
We refer the interested reader to a more detailed discussion of the method to \cite{sieved}.

\smallskip 

In the next section we reduce the inequality (\ref{1.2}) to Theorem \ref{th2}, which, in turn, is reduced to Theorem \ref{th3} in Section \ref{sec5}. The final section is devoted to the proof of Theorem \ref{th3}. 

\section{Greedy sieving using Hypergraph covering}\label{sec4}

Recall that $\S'$ and $\S''$ are the random sets $(S_z-\b)$ and $(S_z-N+\b)$, respectively, where $\b$ is chosen uniformly at random from $\Z/P\Z$. As was mentioned above, by $S'$ and $S''$ we denote their realizations (with respect to some choice of $b$); the same is applied to the random weights $\bl$.

\begin{theorem}\label{th2}
Fix $\d$ satisfying (\ref{1.3}), and suppose that $M-6$, $1/K$, and $\xi-1$ are sufficiently small depending on $\d$, and that $x$ is sufficiently large depending on $\d,M,K,\xi$. Then for any positive $\eps<(M-6)/6$ there exist $b \pmod{P(z)}$ and the sets $\Q'\subseteq\{q\in \Q: q\equiv1\pmod4\}$ and $\Q''\subseteq\{q\in \Q: q\equiv3\pmod4\}$ such that

\medskip

$(i)$ one has 
\begin{equation}\label{3.1}
\big|(S'\cup S'')\cap[-y,y]\big| \leq 9\s y;	
\end{equation} 

$(ii)$ for any $q\in \Q'\cup \Q''$ one has
\begin{equation}\label{3.2}
\sum_{-(K+1)y<n\leq y}\l(H_q;q,n)=\left(1+O\left(\frac{1}{(\log x)^{\d(1+\eps)}}\right)\right)(K+2)y;
\end{equation} 

$(iii)$ for all but at most $\frac{x}{10\log x}$ elements $n$ of $S'\cap[-y,y]$ one has
\begin{equation}\label{3.3}
\sum_{q\in \Q'}\sum_{h\leq KH_q}\l(H_q;q,n-qh)= \left(C_2'+O\left(\frac{1}{(\log x)^{\d(1+\eps)}}\right)\right)(K+2)y,
\end{equation}
and for all but at most $\frac{x}{10\log x}$ elements $n$ of $S''\cap[-y,y]$ one has
\begin{equation}\label{3.4}
	\sum_{q\in \Q''}\sum_{h\leq KH_q}\l(H_q;q,n-qh)= \left(C_2''+O\left(\frac{1}{(\log x)^{\d(1+\eps)}}\right)\right)(K+2)y,
\end{equation}
where $C_2'$ and $C_2''$ are some quantities independent of $n$ with
\begin{equation}\label{3.5}
	10^{2\d} \leq C_2',C_2'' \leq 100.
\end{equation}
 	
\end{theorem}

\smallskip 

Theorem \ref{th2} can be considered as a preparation for the ``greedy'' stage of sieving out the sets $S'$ and $S''$ using large primes from $(z,x/2)$. After fixing an appropriate $b\pmod{P(z)}$ and getting disjoint sets $\Q'$ and $\Q''$ for working with $S'$ and $S''$, respectively, we will be in position to apply the following lemma (which is Lemma 3.1 of \cite{sieved}) to deduce Theorem \ref{th1} from Theorem \ref{th2}.

\begin{lem}[Hypergraph covering lemma]\label{lem3.2} 
Suppose that $0<\d\leq1/2$ and $K\geq1$, and let $y\geq y_0(\d,K)$ with $y_0(\d,K)$ sufficiently large, and let $V$ be a finite set with $|V|\leq y$. Let $1\leq s\leq y$, and suppose that $\e_1,...,\e_s$ are random subsets of $V$ satisfying the following:
\begin{equation}\label{3.6}
|\e_i|\leq \frac{K(\log y)^{1/2}}{\log\log y} \quad (1\leq i\leq s),
\end{equation}
\begin{equation}\label{3.7}
\P(v\in \e_i) \leq y^{-1/2-1/100} \quad (v\in V, 1\leq i\leq s),	
\end{equation}
\begin{equation}\label{3.8}
\sum_{i=1}^s\P(v,v'\in \e_i) \leq y^{-1/2} \quad (v,v'\in V, v\neq v'),	
\end{equation}
\begin{equation}\label{3.9}
\left|\sum_{i=1}^s\P(v\in \e_i)-C_2\right| \leq \eta \quad (v\in V),	
\end{equation}
where $C_2$ and $\eta$ satisfy
\begin{equation}\label{3.10}
10^{2\d} \leq C_2 \leq 100, \quad \eta \geq \frac{1}{(\log y)^{\d}\log\log y} .	
\end{equation}
Then there are subsets $e_i$ of $V$, $1\leq i\leq s$, with $e_i$ being in the support of $\e_i$ for every $i$, and such that
\begin{equation}\label{3.11}
	\left|V\setminus \bigcup_{i=1}^se_i\right| \leq C_3\eta|V|,
\end{equation}
where $C_3$ is an absolute constant.
\end{lem}

\smallskip 

\begin{proof}[Deduction of Theorem \ref{th1} from Theorem \ref{th2}] Let $b$, $\Q'$, and $\Q''$ be from Theorem \ref{th2}. We will apply Hypergraph covering Lemma \ref{lem3.2}. Let
$$
V'=\{n\in S'\cap[-y,y]: (\ref{3.3}) \mbox{ holds} \}
$$
and
$$
V''=\{n\in S''\cap[-y,y]: (\ref{3.4}) \mbox{ holds} \}.
$$
For each $q\in \Q'\cup \Q''$, we define the random integer $\n_q$ by setting
\begin{equation}\label{3.12}
\P(\n_q=n)=\frac{\l(H_q;q,n)}{\sum_{-(K+1)y<m\leq y}\l(H_q; q,m)}
\end{equation}	
Note that by (\ref{3.2}) the denominator is non-zero, so it is well-defined probability distribution.	
Now, for $q\in \Q'$, let 
$$	
\e'_q:=V' \cap \{\n_q+hq: 1\leq h\leq KH_q\},	
$$	
and, for $q\in \Q''$, let 
$$	
\e''_q:=V'' \cap \{\n_q+hq: 1\leq h\leq KH_q\}.	
$$	

We aim to show that there are choices $n_q$ of $\n_q$ such that the corresponding sets $e_q'$ and $e_q''$ obey
\begin{equation}\label{3.13}
\left|V'\setminus \bigcup_{q\in \Q'}e_q'\right| \leq \frac{x}{10\log x}.	
\end{equation}
and 
\begin{equation}\label{3.14}
\left|V''\setminus \bigcup_{q\in \Q''}e_q''\right| \leq \frac{x}{10\log x}.	
\end{equation}
Once it is done, we can set $b\equiv -n_q \pmod{q}$ for $q\in \Q'$ and $b\equiv n_q+N \pmod{q}$ for $q\in Q''$, and make an arbitrary choice of $b\pmod{q}$ for $q\in(z,x/2]\setminus (\Q'\cup \Q'')$. Then, since for any $q\in \Q'$ 
$$
e_q' \subset \{n\in \Z: n\equiv n_q \pmod{q} \},
$$
and for any $q\in \Q''$
$$
e_q'' \subset \{n\in \Z: n\equiv n_q \pmod{q} \},
$$
we get
\begin{equation*}
\Big|\left(S_{x/2}-b\right)\cap[-y,y]\Big| \leq \Big|(S'\cap[-y,y])\setminus V'\Big|+\Big|V'\setminus \bigcup_{q\in \Q'}e_q'\Big|
\leq \frac{x}{10\log x}+\frac{x}{10\log x}=\frac{x}{5\log x},
\end{equation*}
and, similarly,
\begin{equation*}
\Big|\left(S_{x/2}-N+b\right)\cap[-y,y]\Big| \leq \Big|(S''\cap[-y,y])\setminus V''\Big|+\Big|V''\setminus \bigcup_{q\in \Q''}e_q''\Big|
\leq \frac{x}{10\log x}+\frac{x}{10\log x}=\frac{x}{5\log x},
\end{equation*}
and (\ref{1.2}) follows.

We will apply the covering lemma twice: for $V'$ and $V''$, to get (\ref{3.13}) and (\ref{3.14}), respectively. These two applications are completely similar, so we consider only the one concerned with $V'$. We take $s=|\Q'|$, $\{ \e_1,...,\e_s\}=\{\e_q': q\in \Q' \}$, $C_2'$ from Theorem 2, and
$$  
\eta=\frac{1}{100C_3(\log x)^{\d}}.
$$
Then, using (\ref{3.1}), we get
$$
C_3\eta|V'| \leq C_3\eta |S'\cap[-y,y]| \leq \frac{9\s y}{100(\log x)^{\d}}\sim \frac{9y}{100(\log x)^{\d}\log z}\leq \frac{x}{10\log x}
$$
for $x$ large enough, and (\ref{3.13}) follows from (\ref{3.11}). Thus it suffices to verify the conditions of the covering lemma. 

Firstly, by (\ref{2.1}),
$$
|\e'_{q}| \leq KH_q \leq \frac{Ky}{z} \leq \frac{K(\log x)^{1/2}}{\log\log x} \leq \frac{K(\log y)^{1/2}}{\log\log y}, 
$$
so (\ref{3.6}) follows. Further, let $n\in V'$ and $q\in \Q'$. By (\ref{3.12}), (\ref{3.2}), and (\ref{2.1}),
$$
\P(n\in \e_q') = \sum_{h=1}^{KH_q}\P(\n_q=n-qh)	\ll y^{-1}\sum_{h\leq KH_q}\l(H_q; q,n-qh) \ll y^{-1}H_q\s_2^{-KH_q} \leq y^{-9/10},
$$
and (\ref{3.7}) follows. We also have from (\ref{3.3})
$$
\sum_{q\in \Q'}\P(n\in \e_q') = \sum_{q\in \Q'}\sum_{h\leq KH_q}\frac{\l(H_q; q,n-qh)}{\sum_{-(K+1)y<n'\leq y}\l(H_q;q,n')}=C_2'+O\left((\log x)^{-\d(1+\eps)}\right),	
$$
which confirms (\ref{3.9}). Now we turn to (\ref{3.8}). If $v$ and $v'$ both lie in $\e_q$, then $q$ divides $|v-v'|$, which is at most $2y$. But $q>z>\sqrt{2y}$, thus there can be at most one such $q$ for any fixed $v\neq v'$. Therefore (\ref{3.8}) follows from (\ref{3.7}).
 
This completes the proof of (\ref{1.2}), and thus Theorem \ref{th1} follows from Theorem \ref{th2}.
 \end{proof}

\section{The third reduction}\label{sec5}

In this section we deduce Theorem \ref{th2} from the following theorem.

\begin{theorem}\label{th3}
Let $M\geq2$. Then

$(i)$ One has 

\begin{equation}\label{4.1}
\E\Bigl|\S'\cap[-y,y]\Bigr|=\E\Bigl|\S''\cap[-y,y]\Bigr|=\s(2y+1);
\end{equation}
	
\smallskip	
	
$(ii)$ For every $H\in \HH$, every $j\in\{0,1,2\}$, and $i\in\{1,3\}$,	
	
\begin{equation}\label{4.2}
\E\sum_{q\in \Q_{H,i}}\left(\sum_{-(K+1)y<n\leq y}\bl(H;q,n)\right)^j = \left(1+O\left(\frac{1}{H^{M-2}}\right)\right)((K+2)y)^j|\Q_{H,i}|;
\end{equation}	
	
\smallskip		
	
$(iii)$ For every $H\in \HH$ and $j\in\{0,1,2\}$, 
	
\begin{equation}\label{4.3}
\E\sum_{n\in \S'\cap[-y,y]}\left(\sum_{q\in \Q_{H,1}}\sum_{h\leq KH}\bl(H;q,n-qh)\right)^j = \left(1+O\left(\frac{1}{H^{M-2}}\right)\right)\left(\frac{|\Q_{H,1}|\cdot\lfloor KH\rfloor }{\s_2}\right)^j\s (2y+1)	
\end{equation}		
and
\begin{equation}\label{4.4}
\E\sum_{n\in \S''\cap[-y,y]}\left(\sum_{q\in \Q_{H,3} }\sum_{h\leq KH}\bl(H;q,n-qh)\right)^j = \left(1+O\left(\frac{1}{H^{M-2}}\right)\right)\left(\frac{|\Q_{H,3}|\cdot\lfloor KH\rfloor}{\s_2}\right)^j\s (2y+1).
\end{equation}	
\end{theorem}

\bigskip 

We remind the reader that in Theorem \ref{th3} the random variables $\S', \S''$, and $\bl$ are defined in terms of the random variable $\b$ chosen uniformly in $\Z/P\Z$, not the random variables $\n_q$ we used in the
previous section.

\begin{rem}\label{rem4.1}
It can be shown  that	
$$
\E\Bigl|\S'\cap[-y,y]\Bigr|^2=\E\Bigl|\S''\cap[-y,y]\Bigr|^2=\left(1+O\left(\frac{1}{\log y}\right)\right)\s(2y+1);
$$
(and it is actually the relation (4.2) of \cite{sieved}). This implies that both sets $S'\cap[-y,y]$ and $S''\cap[-y,y]$ have size $(2+o(1))\s y$ with probability $1-o(1)$, and thus in fact almost all $\b \pmod{P(z)}$ are good for Theorem \ref{th2}. However, we decided to make the proof slightly shorter and not to provide a proof for the above relation. Thus we use only first moment in (\ref{4.1}) and show that (at least) half of choices of $\b \pmod{P(z)}$ are good for our purpose.  
\end{rem}	

\smallskip 

\begin{proof}[Deduction of Theorem \ref{th2} from Theorem \ref{th3}]

Firstly, we show that (\ref{3.1}) holds with probability at least $1/2$. From (\ref{4.1}) we see that
$$
\E\Bigl|(\S'\cup\S'')\cap[-y,y]\Bigr|\leq \E\Bigl|\S'\cap[-y,y]\Bigr|+ \E\Bigl|\S''\cap[-y,y]\Bigr|=2\s(2y+1),	
$$
and thus from Markov's inequality 
$$
\P\left(\Bigl|(\S'\cup\S'')\cap[-y,y]\Bigr| \geq 4\s(2y+1)\right) \leq 1/2. 
$$
Hence, we have
\begin{equation}\label{4.5}
\Bigl|(\S'\cup \S'')\cap[-y,y]\Bigr| \leq 9\s y	
\end{equation}
with probability at least $1/2$.

Now we work on parts (ii) and (iii) of Theorem 2. Fix $H\in\HH$. From (\ref{4.2}) we have
\begin{equation}\label{4.6}
\E\sum_{q\in \Q_{H,1}}\left(\sum_{-(K+1)y<n\leq y}\bl(H;q,n) - (K+2)y\right)^2 \ll \frac{y^2|\Q_{H,1}|}{H^{M-2}};
\end{equation}	
Now let $\bQ_H'$ be the (random) set of $q\in \Q_{H,1}$ for which
\begin{equation}\label{4.7}
\left|\sum_{-(K+1)y<n\leq y}\bl(H;q,n) - (K+2)y\right| \leq \frac{y}{H^{1+\eps}};
\end{equation}	
By estimating the left-hand side of (\ref{4.6}) from below by the sum over $q\in \Q_{H,1}\setminus\bQ'_H$, we find that
\begin{equation}\label{4.8}
\E|\Q_{H,1}\setminus\bQ'_H|	\ll \frac{|\Q_{H,1}|}{H^{M-4-2\eps}}.
\end{equation} 
Now we set
$$
\bQ'=\bigcup_{H\in\HH}\bQ_H' \subseteq \{q\in\Q: q\equiv1\pmod4 \}.
$$
In completely similar way we define the random set 
$$
\bQ''=\bigcup_{H\in\HH}\bQ_H'' \subseteq \{q\in\Q: q\equiv3\pmod4 \},
$$
where, for each $H\in\HH$, we denote by $\bQ_H''$ the random set of $q\in \Q_{H,3}$ for which
\begin{equation}\label{4.9}
\left|\sum_{-(K+1)y<n\leq y}\bl(H;q,n) - (K+2)y\right| \leq \frac{y}{H^{1+\eps}};
\end{equation}	
again we have 
\begin{equation}\label{4.10} 
\E|\Q_{H,3}\setminus\bQ'_H|	\ll \frac{|\Q_{H,3}|}{H^{M-4-2\eps}}.
\end{equation}
for all $H\in\HH$.

Now we turn to the condition (iii) of Theorem \ref{th2}. Fix $H$. Similarly to (\ref{4.6}), from (\ref{4.3}) we have
\begin{equation}\label{4.11}
\E\sum_{n\in \S'\cap[-y,y]}\left(\sum_{q\in \Q_{H,1}}\sum_{h\leq KH}\bl(H;q,n-qh)-\frac{|\Q_{H,1}|\cdot\lfloor KH\rfloor}{\s_2}\right)^2 \ll \frac{1}{H^{M-2}}\left(\frac{|\Q_{H,1}|\cdot\lfloor KH\rfloor}{\s_2}\right)^2\s y.	
\end{equation}		
Let $\bmE'_H$ be the set of $n\in \S'\cap[-y,y]$ such that
\begin{equation}\label{4.12}
\left|\sum_{q\in \Q_{H,1}}\sum_{h\leq KH}\bl(H;q,n-qh)-\frac{|\Q_{H,1}|\cdot\lfloor KH\rfloor}{\s_2}\right| \geq \frac{|\Q_{H,1}|\cdot\lfloor KH\rfloor}{\s_2H^{1+\eps}}.	
\end{equation}		
Then, since $M>6$ and $\eps$ is small, (\ref{4.11}) implies that
$$
\E|\bmE'_H| \ll \frac{\s y}{H^{1+2\eps}},
$$
and, hence, $|\bmE'_H|\leq \frac{\s y}{H^{1+\eps}}$ with probability $1-O(H^{-\eps})$.

Now we estimate the contribution from ``bad'' primes $q\in \Q_{H,1}\setminus \bQ'_H$. For any $h\leq KH$, we get from Cauchy-Schwarz inequality (for vector functions)
\begin{multline*}
\E\sum_{q\in \Q_{H,1}\setminus\bQ'_H}\sum_{n\in \S'\cap[-y,y]}\bl(H;q,n-qh) \\
\leq\left(\E|\Q_{H,1}\setminus\bQ'_H|\right)^{1/2}\left(\sum_{\Q_{H,1}\setminus\bQ'_H}\left|\sum_{-(K+1)y<n\leq y}\bl(H;q,n)\right|^2\right)^{1/2}, 	
\end{multline*}
where we extended the range of summation of $\bl(H;q,\cdot)$ to the larger interval $(-(K+1)y,y]$ (note that the weights $\bl(H;q,\cdot)$ are non-negative). Further, by the triangle inequality, (\ref{4.6}), and (\ref{4.8}), 
\begin{multline*}
\E\sum_{\Q_{H,1}\setminus\bQ'_H}\left|\sum_{-(K+1)y<n\leq y}\bl(H;q,n)\right|^2 \\
\leq2\E\sum_{\Q_{H,1}\setminus\bQ'_H}\left(\left|\sum_{-(K+1)y<n\leq y}\bl(H;q,n)-(K+2)y\right|^2+(K+2)^2y^2\right) \ll \frac{y^2|\Q_{H,1}|}{H^{M-4-2\eps}}.	
\end{multline*}
Combining two latter estimates (and using (\ref{4.8} again), we get, after summing over all $h\leq KH$,
$$
\E\sum_{n\in \S'\cap[-y,y]}\sum_{q\in \Q_{H,1}\setminus \bQ'_H}\sum_{h\leq KH} \bl(H;q,n-qh) \ll \frac{y|\Q_{H,1}|}{H^{M-5-2\eps}}.
$$
Let $\bmF'_H$ be the set of $n\in\S'\cap[-y,y]$ such that
\begin{equation}\label{4.13} 
\sum_{q\in \Q_{H,1}\setminus \bQ'_H}\sum_{h\leq KH} \bl(H;q,n-qh) \geq \frac{|\Q_{H,1}|\cdot\lfloor KH\rfloor}{\s_2H^{1+\eps}}.	
\end{equation} 
Then 
$$
\E|\bmF'_H| \ll \frac{\s_2y}{H^{M-5-3\eps}} \ll \frac{\s y \log H}{H^{M-5-3\eps}}
$$ 
and by Markov's inequality
$$
|\bmF'_H| \leq \frac{\s y}{H^{1+\eps}} 
$$
with probability $1-O(H^{-(M-6-5\eps)})$. Since $\eps<(M-6)/6$, we have $M-6-5\eps>\eps$, and the last probability becomes $1-O(H^{-\eps})$.

Analogously, using (\ref{4.4}), we can define the set $\bmE''_H$ of $n\in \S''\cap[-y,y]$ with
\begin{equation}\label{4.14}
\left|\sum_{q\in \Q_{H,3}}\sum_{h\leq KH}\bl(H;q,n-qh)-\frac{|\Q_{H,3}|\cdot\lfloor KH\rfloor}{\s_2}\right| \geq \frac{|\Q_{H,3}|\cdot\lfloor KH\rfloor}{\s_2H^{1+\eps}},	
\end{equation}		 
and the set $\bmF''_H$ of $n\in \S''\cap[-y,y]$ with 
\begin{equation}\label{4.15} 
\sum_{q\in \Q_{H,3}\setminus \bQ''_H}\sum_{h\leq KH} \bl(H;q,n-qh) \geq \frac{|\Q_{H,3}|\cdot\lfloor KH\rfloor}{\s_2H^{1+\eps}};	
\end{equation} 
 we have 
$$ 
|\bmE''_H|, |\bmF''_H|  \leq \frac{\s y}{H^{1+\eps}}
$$
with probability $1-O(H^{-\eps})$. Since $\sum_{H\in\HH}H^{-\eps}\ll (\log x)^{-\d\eps}$, we see that the probability that there is $H\in\HH$ such that least one the sets $\bmE'_H, \bmF'_H, \bmE''_H, \bmF''_H$ has size greater than $(\s y)H^{-1-\eps}$ is $o(1)$.

Now we are ready to make a choice of $\b \pmod{P(z)}$. We consider the event that (\ref{4.5}) holds and that for each $H\in\HH$, all the four sets $\bmE'_H, \bmF'_H, 
\bmE''_H, \bmF''_H $ have size at most $(\s y)H^{-1-\eps}$. By the above discussion, this event holds with probability at least $1/2-o(1)$). From now, we fix a $\b\pmod{P(z)}$ such that it is so, and thus all of our random sets and weights become deterministic.  

Let
$$
\mathcal{N}'=S'\cap[-y,y]\setminus\bigcup_{H\in\HH}\left(\mE'_H\cup\mF'_H\right)
$$
and
$$
\mathcal{N}''=S''\cap[-y,y]\setminus\bigcup_{H\in\HH}\left(\mE''_H\cup\mF''_H\right).
$$
We verify (\ref{3.3}) for $n\in\mathcal{N}'$, and (\ref{3.4}) will follow from our construction in absolutely similar way. The number of exceptional elements satisfies
$$
\left|\bigcup_{H\in\HH}\left(\mE'_H\cup\mF'_H\right)\right| \leq \frac{\s y}{(\log x)^{(1+\eps)\d}},
$$
which is smaller than $\frac{x}{10\log x}$ for large $x$. We fix arbitrary $n\in \mathcal{N}'$. For such $n$, the inequalities opposite to (\ref{4.12}) and (\ref{4.13}) hold, and therefore for each $H\in\HH$,
$$
\sum_{q\in \Q_{H,1}}\sum_{h\leq KH}\l(H;q,n-qh)=\left(1+O\left(\frac{1}{(\log x)^{(1+\eps)\d}}\right)\right) \frac{|\Q_{H,1}|\cdot\lfloor KH\rfloor}{\s_2} 
$$
due to our choice of $M$.
Summing over all $H\in\HH$, we have
\begin{equation}\label{4.16}
\sum_{q\in \Q'}\sum_{h\leq KH}\l(H;q,n-qh)=\left(1+O\left(\frac{1}{(\log x)^{(1+\eps)\d}}\right)\right)C_2'(K+2)y
\end{equation}
with (recall that $\s_2=\s_2(H)$)
$$
C_2'=\frac{1}{(K+2)y}\sum_{H\in\HH} \frac{|\Q_{H,1}|\cdot\lfloor KH\rfloor}{\s_2}.
$$
Note that $C_2'$ depends on $x,K,M,\xi,$ and $\d$, but not on $n$. Since 
$$
\lfloor KH \rfloor=KH(1+O(1/H))=KH(1+O(\log x)^{-\d})
$$
and
$$
\s_2^{-1}=\prod_{H^M<p\leq z}(1-1/p)^{-1}\sim\frac{\log z}{M\log H},
$$
we get, using (\ref{2.5}), 
$$
C_2'\sim \frac{K}{(K+2)y} \cdot \frac12(1-1/\xi)\sum_{H\in\HH}\frac{y/H}{\log x}\cdot\frac{H\log z}{M\log H} \sim \frac{K(1-1/\xi)}{2M(K+2)}\sum_{H\in\HH}\frac{1}{\log H},
$$
as $x\to\infty$. Recalling the definition of $\HH$, we see that
$$
C_2'\sim \frac{K(1-1/\xi)}{2M(K+2)\log \xi}\sum_{j}\frac{1}{j},
$$
where $j$ runs over the interval
$$
\frac{\d\log\log x}{\log \xi} \leq j \leq \frac{(1/2+o(1))\log\log x}{\log \xi}.
$$
We thus obtain
$$
C_2'\sim \frac{K(1-1/\xi)}{2M(K+2)\log \xi}\log(1/(2\d)).
$$
Recall the condition (\ref{1.3}) on $\d$ and the fact that $K$ and $x$ are sufficiently large, $M$ is sufficiently close to $6$, $\xi$ is close enough to $1$ (all in terms of $\d$). Then we see that
$$
10^{2\d}\leq C_2'\leq 100.
$$
This together with (\ref{4.16}) implies (\ref{3.3}). Arguing similarly, one can obtain (\ref{3.4}) for $n\in \mathcal{N}''$. The claim follows.  
\end{proof}

It remains to establish Theorem \ref{th3}. This is the aim of the last section of the paper.

\section{Computing correlations}\label{sec6}

We first introduce some notation. For $H\in\HH$, let $\D_H$ be the collection of square-free numbers $d$, all of whose prime divisors lie in $(H^M,z]$.
Further, for $A>0$, let 
\begin{equation}\label{5.1}
E_A(m;H)=\sum_{d\in \D_H\setminus\{1\}}\frac{A^{\o(d)}}{d}1_{m \equiv 0\pmod d}.	
\end{equation}
Note that $E_A(m;H)=E_A(-m;H)$.

\bigskip 

We need the following two lemmas (see Lemmas 5.1 and 5.2 of \cite{sieved}).

\begin{lem}\label{lem5.1} 
Let $10<H<z^{1/M}$, $1\leq l\leq 10KH$, and $\U\subset \V$ be two finite sets with $|\V|=l$. Then 	
$$
\P(\U\subset \S'_2)=\P(\U\subset \S''_2)=\s_2^{|\U|}\Biggl(1+O\Bigl(|\U|^2H^{-M}+l^{-2}\sum_{\substack{v,v'\in\V \\ v\neq v'}}E_{2l^2}(v-v';H)\Bigr)\Biggr).
$$		
\end{lem}

\begin{rem}
In \cite{sieved} Lemma \ref{lem5.1} is in fact formulated for the probability $\P(\U\subset \S_2)$, where $\S_2=S_{H^M,z}+\b_2$, but it does make any difference for us since it is easy to see (by making a change of variables $\b_2\mapsto -\b_2$ or $\b_2\mapsto -N+\b_2$) that $\P(\U\subset \S_2)=\P(\U\subset \S'_2)=\P(\U\subset \S''_2)$.
\end{rem}

\begin{lem}\label{lem5.2} 
Let $10<H<z^{1/M}$ and $(m_t)_{t\in T}$ be a finite sequence such that
\begin{equation}\label{5.2}
\sum_{t\in T}1_{m_t\equiv a\pmod{d}} \ll \frac{X}{\varphi(d)}+R	
\end{equation}
for some $X,R>0$, and all $d\in \D_H\setminus\{1\}$ and $a\in\Z/d\Z$. Then for any $A$ with $0<A\leq H^M$ and any integer $j$	
$$
\sum_{t\in T}E_A(m_t+j;H) \ll \frac{XA}{H^M}+R\exp(A\log\log y).
$$
\end{lem}

\medskip 

Now we are ready to prove Theorem 3.

\begin{proof}[Proof of Theorem \ref{th3} (i)] First, we notice that by making a change of variables, it is easy to see that	
$$
\E\Big|\S'\cap[-y,y]\Big|=\E\Big|\S''\cap[y,-y]\Big|.
$$
Thus, it suffices to prove the claims concerned with $\S'$. By linearity of expectation, we get
$$
\E\Big|\S'\cap[y,-y]\Big|=\sum_{-y\leq n\leq y}\P(n\in \S').
$$	
Now, since $\b$ is taken uniformly from $\Z/P(z)\Z$, for each fixed $n$ we have from the Chinese Remainder Theorem
$$
\P(n\in \S')=\prod_{p\leq z}\P(\b\notequiv -n \pmod{p})=\s,
$$
and (\ref{4.1}) follows.  
\end{proof}

\begin{proof}[Proof of Theorem \ref{th3} (ii)] We prove the claim for $i=1$ (the $i=3$ case can be handled similarly). Fix $H\in\HH$. The case $j=0$ is trivial, and we turn to $j=1$, which is
\begin{equation}\label{5.3}
\E\sum_{q\in \Q_{H,1}}\sum_{-(k+1)y<n\leq y}\bl(H; q,n)=
\left(1+O\left(\frac{1}{H^{M-2}}\right)\right)(K+2)y|\Q_{H,1}|. 
\end{equation}	 
By (\ref{2.11}), the left-hand side expands as
$$
\E\sum_{q\in\Q_{H,1}}\sum_{-(K+1)y<n\leq y}\frac{1_{\bAP'(KH;q,n)\subset\S'_2}}{\s_2^{|\bAP'(KH;q,n)|}}.
$$
Recall that, according to the definitions (\ref{2.8}) and (\ref{2.9}), $\b_1$ and $\b_2$ are independent, and so are $\bAP'(KH; q,n)=\{n+qh: 1\leq h\leq KH\}\cap \S'_1$ and $\S'_2$. Then the above expression equals
$$  
\sum_{q\in\Q_{H,1}}\sum_{-(K+1)y<n\leq y}\sum_{b_1\pmod {P_1}}\frac{\P(\b_1=b_1)}{\s_2^{|\AP'(KH;q,n)|}}\P(\AP'(KH; q,n)\subset\S'_2).
$$
For fixed $q$, $n$, and $b_1$, we apply Lemma \ref{lem5.1} to the (deterministic) sets $\U=\AP'(KH;q,n)$ and $\V=\{n+qh: 1\leq h\leq KH\}$, and find that the left-hand side of (\ref{5.3}) is equal to
$$
\sum_{q\in \Q_{H,1}}\sum_{-(K+1)y<n\leq y}\left(1+O\left(H^{-(M-2)}+H^{-2}\sum_{\substack{1\leq h<h'\leq KH}}E_{2K^2H^2}(qh-qh';H) \right)\right).
$$
Now it is enough to show that, for any $1\leq h< h'\leq KH$,
$$
\sum_{q\in \Q_{H,1}}E_{2K^2H^2}(qh-qh';H) \ll \frac{|\Q_{H,1}|}{H^{M-2}}.
$$
For future reference, we show the more general bound
\begin{equation}\label{5.4}
\sum_{q\in Q_{H,i}}	E_{8K^2H^2}(qr+s;H) \ll \frac{|Q_{H,i}|}{H^{M-2}}
\end{equation}
for each $i\in\{1,3\}$, $0<|r|\leq KH$, and any integer $s$. Note that $E_A(m;H)$ is an increasing function of $A$.

To show (\ref{5.4}), we fix $r$ and $s$. For any $d\in\D_{H\setminus\{1\}}$, all prime divisors of $d$ are greater than $H^M>KH\leq |r|$, and so $r$ and $d$ are coprime. Hence, the congruence $qr\equiv a\pmod{d}$ holds for at most one residue class $q\pmod d$. Therefore, for $d\leq y^{1/2}$, we have by Brun-Titchmarch inequality (recall that $H\leq (\log y)^{1/2}$ by (\ref{2.4}))
$$
\#\{q\in \Q_{H,i}: qr\equiv a\pmod d\} \ll \frac{y/H}{\varphi(d)\log(y/d)} \ll \frac{y/H}{\varphi(d)\log y}.
$$
For $d>y^{1/2}$, we can forget that $q$ is restricted to be prime and trivially get
$$
\#\{q\in \Q_{H,i}: qr\equiv a\pmod d\} \ll \frac{y/H}{d} + 1 \ll \frac{y^{1/2}}{H}.
$$
So for each $d$ we have
$$
\#\{q\in \Q_{H,i}: qr\equiv a\pmod d\} \ll \frac{y/H}{\varphi(d)\log y}+\frac{y^{1/2}}{H}.
$$
Hence, by Lemma \ref{lem5.2} we get, using (\ref{2.4}) again,
$$
\sum_{q\in Q_{H,i}}	E_{8K^2H^2}(qr+s;H) \ll \frac{y/H}{\log y}\frac{H^2}{H^M}+\frac{y^{1/2}}{H}\exp(O(H^2\log\log y)) \ll \frac{|\Q_{H,i}|}{H^{M-2}},	
$$
since $|\Q_{H,i}|\asymp \frac{y/H}{\log y}$ for each $i\in\{1,3\}$. So, (\ref{5.4}) is proved, and
thus the case $j=1$ of Theorem \ref{th3} (ii) follows.

Now we turn to the case $j=2$ of (ii), which is 
$$
\E\sum_{q\in \Q_{H,1}}\left(\sum_{-(K+1)y<n\leq y}\bl(H; q,n)\right)^2=
\left(1+O\left(\frac{1}{H^{M-2}}\right)\right)(K+2)^2y^2|\Q_{H,1}|. 
$$ 
The left-hand side is expanded as
$$
\E\sum_{q\in\Q_{H,1}}\sum_{-(K+1)y<n_1,n_2\leq y}\bl(H;q,n_1)\bl(H;q,n_2). 
$$ 
We first note that, for each fixed $q$, the contribution of the pairs $(n_1,n_2)$ for which $|n_1-n_2|\leq KH$ is negligible: indeed, there are $O(yH)$ such pairs, and each of them contributes at most $\s_2^{-2KH}=y^{o(1)}$, so the total contribution of such pairs is $O(y^{1+o(1)}|\Q_{H,1}|)$. Thus we may restrict our attention to those pairs $(n_1,n_2)$ for which the sets $\{n_\nu+qh: 1\leq h\leq KH\}$, $\nu=1,2$, do not intersect; let us call these pairs \textit{good}. Then it is enough to show that
\begin{multline}\label{5.5}
\E\sum_{q\in\Q_{H,1}}\sum_{\substack{-(K+1)y<n_1,n_2\leq y\\(n_1,n_2) \text {good}}}\frac{1_{\bAP'(KH;q,n_1)\cup \bAP'(KH;q,n_2)\subset\S'_2}}{\s_2^{|\bAP'(KH;q,n_1)|+|\bAP'(KH;q,n_2)|}}\\
=\left(1+O\left(\frac{1}{H^{M-2}}\right)\right)(K+2)^2y^2|\Q_{H,1}|.
\end{multline}
Arguing as in the case $j=1$, for any realization $b_1$ of $\b_1$ and any good pair $(n_1,n_2)$, we can apply Lemma \ref{lem5.1} with
$$
\U=\AP'(KH;q,n_1)\sqcup \AP'(KH;q,n_2)
$$ 
and 
$$
\V=\{n_1+qh: 1\leq h\leq KH\}\sqcup 
\{n_2+qh: 1\leq h\leq KH\}.
$$
Then, since 
$$
|\U|=|\AP'(KH;q,n_1)|+|\AP'(KH;q,n_2)|,
$$
and $|\V|= 2\lfloor KH\rfloor$, we see that the left-hand side of (\ref{5.5}) equals 
\begin{multline*} 
\sum_{q\in\Q_{H,1}}\sum_{\substack{-(K+1)y<n_1,n_2\leq y\\(n_1,n_2) \text {good}}}\Biggl(1+O\Biggl(\frac{1}{H^{M-2}}+H^{-2}\sum_{1\leq h,h'\leq KH}1_{h\neq h'}E_{8K^2H^2}(qh-qh';H)+\\
+1_{n_1\neq n_2}E_{8K^2H^2}(n_1-n_2+qh-qh';H)\Biggr)\Biggr).
\end{multline*}
Recalling that all but $O(yH)$ pairs $(n_1,n_2)$ are good, we get the main term $(K+2)^2y^2|\Q_{H,1}|$ from here. We also obtain acceptable error terms using (\ref{5.4}), except for the summands with $h=h'$. To handle them, we note that for any fixed $n_2$, any positive integer $d$ and $a\pmod{d}$, 
$$
\#\{-(K+1)y<n_1\leq y: n_1-n_2 \equiv a\pmod{d}\} \ll \frac{y}{d}+1.
$$
Thus, by Lemma \ref{lem5.2}
$$
\sum_{-(K+1)y<n_1, n_2\leq y}E_{8K^2H^2}(n_1-n_2;H) \ll \frac{y^2}{H^{M-2}}+y\exp(O(H^2\log\log y)) \ll \frac{y^2}{H^{M-2}}
$$
by (\ref{2.4}). Thus, (\ref{5.5}) and the claim for $j=2$ follow. 
\end{proof}

\begin{proof}[Proof of Theorem \ref{th3} (iii)] Fix $H$. We prove only (\ref{4.3}), since (\ref{4.4}) can be handled in absolutely similar manner. The case $j=0$ follows from part (i) (that is, (\ref{4.1})), so we focus on the case $j=1$, which is
$$
\E\sum_{n\in \S'\cap[-y,y]}\sum_{q\in \Q_{H,1}}\sum_{h\leq KH}\bl(H;q,n-qh) = \left(1+O\left(\frac{1}{H^{M-2}}\right)\right)|\Q_{H,1}|\cdot\lfloor KH\rfloor\s_1 (2y+1).		
$$	
It is enough to show that, for any $h\leq KH$,
\begin{equation}\label{5.6}
\E\sum_{n\in \S'\cap[-y,y]}\sum_{q\in \Q_{H,1}}\bl(H;q,n-qh) = \left(1+O\left(\frac{1}{H^{M-2}}\right)\right)|\Q_{H,1}|\s_1 (2y+1).
\end{equation}
According to (\ref{2.11}), the left-hand side is equal to
$$
\E\sum_{n\in\S'\cap[-y,y]}\sum_{q\in \Q_{H,1}}\frac{1_{\bAP'(H;q,n-qh)\subset \S'_2}}{\s_2^{|\bAP'(H;q,n-qh)|}}
$$
By (\ref{2.85}), the condition $n\in \S'\cap[-y,y]$ implies that $n\in\S'_1\cap[-y,y]$. On the other hand, if $n\in\S'_1$, then $n\in\bAP'(H;q,n-qh)$, and thus the condition $n\in\S'_2$ is contained in the condition $ \AP'(H;n,n-qh)\subset\S'_2$. So the left-hand side of (\ref{5.6}) can be rewritten as
$$
\E\sum_{n\in\S'_1\cap[-y,y]}\sum_{q\in \Q_{H,1}}\frac{1_{\bAP'(H;q,n-qh)\subset \S'_2}}{\s_2^{|\bAP'(H;q,n-qh)|}}.
$$
Recalling that $\S_2'$ is independent of $\S_1'$ and of $\bAP'(H;q,n-qh)$, we may apply Lemma \ref{lem5.1} as before and find that the left-hand side of (\ref{5.6}) is
$$
\E\sum_{n\in\S'_1\cap[-y,y]}\sum_{q\in \Q_{H,1}}\Bigg(1+O\Bigg(\frac{1}{H^{M-2}}+H^{-2}\sum_{\substack{h',h''\leq KH\\HH'\neq h''}}E_{2K^2H^2}(qh'-qh'')\Bigg)\Bigg).
$$
Now since
\begin{equation}\label{5.7}
\E\Big|\S'_1\cap[-y,y]\Big|=\s_1(2y+1),	
\end{equation}
we see that (\ref{5.6}) follows from (\ref{5.4}).

\medskip

Now we turn to the case $j=2$ of (iii), which is 
\begin{multline*}
\sum_{h_1,h_2\leq KH}\E\sum_{n\in\S'_1\cap[-y,y]}\sum_{q_1,q_2\in\Q_{H,1}}\bl(H;q_1,n-q_1h_1)\bl(H;q_2,n-q_2h_2)\\
=\left(1+O\left(\frac{1}{H^{M-2}}\right)\right)|\Q_{H,1}|^2\cdot\lfloor KH\rfloor^2\frac{\s_1}{\s_2}y.	
\end{multline*}	
By (\ref{2.12}), the left-hand side is 
\begin{equation}\label{5.8}
\sum_{h_1,h_2\leq KH}\E\sum_{n\in\S'\cap[-y,y]}\sum_{q_1,q_2\in \Q_{H,1}}\frac{1_{\bAP'(H;q_1,n-q_1h_1)\cup \bAP'(H;q_2,n-q_2h_2)\subset \S'_2}}{\s_2^{|\bAP'(H;q_1,n-q_1h_1)|+|\bAP'(H;q_2,n-q_2h_2)|}}
\end{equation}
Note that by (\ref{2.4}) and (\ref{5.7}) the contribution from $q_1=q_2$ is
$$ 
\ll H^2\s_2^{-2KH}|\Q_{H,1}|\s_1y\leq |\Q_{H,1}|^2y^{o(1)}, 
$$
which is an acceptable error term. Now if $q_1\neq q_2$, then the set
$$
\bAP'(H;q_1,n-q_1h_1)\cup \bAP'(H;q_2,n_2-q_2h_2)
$$
has size $|\bAP'(H;q_1,n-q_1h_1)|+|\bAP'(H;q_2,n-q_2h_2)|-1$, 
since $n$ is the only common element of these two progression (recall that $q_1$ and $q_2$ are primes $\gg y/H$ and $h_1,h_2\ll H=y^{o(1)}$). As before, we can take the summation in (\ref{5.8}) over $n\in\S'_1\cap[-y,y]$, and then apply Lemma \ref{lem5.1} to rewrite the terms in (\ref{5.8}) with $q_1\neq q_2$ as
$$
\lfloor KH\rfloor^2\s_2^{-1}\E\sum_{n\in\S'_1\cap[-y,y]}\sum_{\substack{q_1,q_2\in \Q_{H,1}\\ q_1\neq q_2}}\left(1+O\left(\frac{1}{H^{M-2}}+\frac{E'(q_1)+E'(q_2)+E''(q_1,q_2)}{H^2}\right)\right),
$$
where 
$$
E'(q)=\sum_{\substack{h,h'\leq KH \\ h\neq h'}} E_{8K^2H^2}(qh-qh';H)
$$
and
$$
E''(q_1,q_2)=\sum_{\substack{h_1',h_2'\leq KH \\ h_1'\neq h_1, h_2'\neq h_2}} E_{8K^2H^2}(q_1h_1'-q_1h_1-q_2h_2'+q_2h_2;H)
$$
The contribution of $E'(q_1)+E'(q_2)$ is acceptably small as we already saw in the proof of the case $j=1$. Finally, we need to show that
$$
\sum_{q_1,q_2\in\Q_{H,1}}E_{8K^2H^2}(q_1h_1'-q_1h_1-q_2h_2'+q_2h_2;H) \ll \frac{|\Q_{H,1}|}{H^{M-2}}
$$
for each $h_1',h_2'$ with $h_1'\neq h_1$ and $h_2'\neq h_2$. It follows from (\ref{5.4}) applied to the sum over $q_1$ with $r=h_1'-h_1$ and $s=-q_2h_2'+q_2h_2$ (and then summing over all $q_2$). This completes the proof of the case $j=2$, and Theorem \ref{th3} follows.
\end{proof}

\end{document}